\documentclass[11pt]{amsart}
\usepackage{graphicx}
\newtheorem{thm}{Theorem}[section]

\newtheorem{lem}[thm]{Lemma}

\theoremstyle{definition}
\newtheorem{defn}[thm]{Definition}
\theoremstyle{remark}

\theoremstyle{conclusion}

\numberwithin{equation}{section}

\begin{document}
\title[Continuous dependence for $H^{2}$ critical NLS]
{ Continuous dependence for $H^{2}$ critical nonlinear Schr\"{o}dinger equations in high dimensions }

\author{Wei Dai}

\address{Institute of Applied Mathematics, AMSS, Chinese Academy of Sciences, Beijing 100190, P. R. China}
\email{daiwei@amss.ac.cn}

\begin{abstract}
The global existence of solutions in $H^{2}$ is well known for $H^{2}$ critical nonlinear Schr\"{o}dinger equations with small initial data in high dimensions $d\geq8$(see \cite{C-W}). However, even though the solution is constructed by a fixed-point technique, continuous dependence in $H^{2}$ does not follow from the contraction mapping argument. Comparing with the low dimension cases $4<d<8$, there is an obstruction to this approach because of the sub-quadratic nature of the nonlinearity(which makes the derivative of the nonlinearity non-Lipschitz). In this paper, we resolve this difficulty by applying exotic Strichartz spaces of lower order instead and show that the solution depends continuously on the initial value in the sense that the local flow is continuous $H^{2}\rightarrow H^{2}$.
\end{abstract}
\maketitle{\small {\bf Keywords:} Nonlinear Schr\"{o}dinger equation; Continuous dependence; Cauchy problem; Strichartz's estimates. \\

{\bf 2010 MSC} Primary: 35Q55; Secondary: 35B30, 46E35.}

\section{INTRODUCTION}

In this paper, we study the Cauchy problem for the following nonlinear Schr\"{o}dinger equation
\begin{equation}\label{eq1}
\left\{
  \begin{array}{ll}
    i\partial_{t}u+ \Delta u=g(u), \,\,\, t\in \mathbb{R}, \,\, x\in \mathbb{R}^{d},\\
    u(0,x)=\phi(x)\in H^{2}(\mathbb{R}^{d}), \,\,\, x\in \mathbb{R}^{d},
  \end{array}
\right.
\end{equation}
where the spatial dimension $d\geq8$, the $H^{2}$-critical nonlinearity $g$ satisfies
\begin{equation}\label{nonlinearity0}
    g\in C^{1}(\mathbb{C},\mathbb{C}), \,\,\,\,\, g(0)=0,
\end{equation}
and obeys the power-type estimates
\begin{equation}\label{nonlinearity1}
    g_{z}(u), \,\, g_{\bar{z}}(u)=O(|u|^{\frac{4}{d-4}}),
\end{equation}
\begin{equation}\label{nonlinearity2}
    g_{z}(u)-g_{z}(v), \,\, g_{\bar{z}}(u)-g_{\bar{z}}(v)=O(|u-v|^{\frac{4}{d-4}}),
\end{equation}
where $O$ denotes the Landau's symbol and $g_{z}$, $g_{\bar{z}}$ are the usual complex derivatives be defined under the identification $\mathbb{C}=\mathbb{R}^{2}$, that is,
\[g_{z}:=\frac{1}{2}(\frac{\partial g}{\partial x}-i\frac{\partial g}{\partial y}), \,\,\,\, g_{\bar{z}}:=\frac{1}{2}(\frac{\partial g}{\partial x}+i\frac{\partial g}{\partial y}).\]

The Cauchy problem (\ref{eq1}) in the Sobolev space $H^{2}(\mathbb{R}^{d})$ has been quite extensively studied(see \cite{B-G,C-D-Y,C-W,C,K1,K2,K3,P,T3}). The local well-posedness for IVP (\ref{eq1}) was established in \cite{C-W,C,K3} in both $H^{2}$ subcritical and critical cases by using spatial derivatives, in which we would need a further assumption $\frac{4}{d-4}>1$. The reason for the additional restriction on spatial dimension $d$ is that obtaining $H^{2}$ estimates by differentiating twice the equation in space would require that the nonlinearity $g$ is sufficiently smooth at the origin. Nevertheless, we may as well differentiate the equation once in time, and then deduce $H^{2}$ estimates by the equation. Based on this idea which goes back to \cite{K1}, the authors were able to construct solutions to (\ref{eq1}) in $H^{2}$-subcritical cases by using time derivatives with minimal regularity assumptions on the nonlinearity $g$ in \cite{C-W,C,K1,K2,P,T3}. In the $H^{2}$-critical case, global well-posedness for IVP (\ref{eq1}) in arbitrary dimension $d>4$ was also obtained in \cite{C-W,P} provided $\|\phi\|_{\dot{H}^{2}}$ is sufficiently small. As for the continuity of the solution map $\phi\mapsto u$ for (\ref{eq1}), it's well known that we can deduce from contraction mapping arguments and interpolation inequalities that the local solution flow is continuous $H^{s}\rightarrow H^{s}$ for any $s<2$(see \cite{C-W,C,K3}). But it is weaker than the expected one, that is, continuous dependence on initial data in $H^{2}(\mathbb{R}^{d})$. To our best knowledge of previous results, continuous dependence in $H^{2}(\mathbb{R}^{d})$ was proved in \cite{C-D-Y,K1} for the subcritical case and in \cite{C-D-Y} for the critical case and low spatial dimensions $d<8$.

In this paper, by using exotic Strichartz spaces equipped with lower fractional order time derivatives, we show that continuous dependence holds in $H^{2}(\mathbb{R}^{d})$ in the standard sense for the $H^{2}$-critical Cauchy problem (\ref{eq1}) in high dimensions $8\leq d<12$. More precisely, our main result is the following.

\begin{thm}\label{th1}
Assume $d\geq8$ and the nonlinearity $g$ satisfies conditions \eqref{nonlinearity0}-\eqref{nonlinearity2}. If initial value $\phi\in H^{2}(\mathbb{R}^{d})$ with $\|\Delta \phi\|_{L^{2}(\mathbb{R}^{d})}$ sufficiently small, then by Theorem 1.4 in \cite{C-W}, there exists a unique, global, strong $H^{2}$-solution $u\in C(\mathbb{R},H^{2}(\mathbb{R}^{d}))$ of the Cauchy problem \eqref{eq1}. Moreover,
\[u\in L^{q}(\mathbb{R},H^{2,r}(\mathbb{R}^{d})) \,\,\,\,\,\, and \,\,\,\,\,\, u_{t}\in L^{q}(\mathbb{R},L^{r}(\mathbb{R}^{d}))\]
for every admissible pair $(q,r)$. If we assume further $8\leq d<12$, then the solution $u$ depends continuously on the initial value $\phi$, that is, if $\phi_{n}\rightarrow \phi$ in $H^{2}(\mathbb{R}^{d})$ and if $u_{n}$ denotes the solution of \eqref{eq1} with the initial value $\phi_{n}$, then $u_{n}$ is global for $n$ large enough, furthermore, $u_{n}\rightarrow u$ in $L^{q}(\mathbb{R},H^{2,r}(\mathbb{R}^{d}))$ and $(u_{n})_{t}\rightarrow u_{t}$ in $L^{q}(\mathbb{R},L^{r}(\mathbb{R}^{d}))$ as $n\rightarrow\infty$ for every admissible pair $(q,r)$. In particular, $u_{n}\rightarrow u$ in $C(\mathbb{R},H^{2}(\mathbb{R}^{d}))$.
\end{thm}

The rest of this paper is organized as follows. In Section 2 we give some useful notation and preliminary knowledge. Section 3 is devoted to the proof of Theorem \ref{th1}.

\section{Notation and preliminary knowledge}
\subsection{Some notation}
Throughout this paper, we use the following notation. $\bar{z}$ is the conjugate of the complex number $z$, $\Re z$ and $\Im z$ are respectively the real and the imaginary part of the complex number $z$. We denote by $p'$ the conjugate of the exponent $p\in[1,\infty]$ defined by $\frac{1}{p}+\frac{1}{p'}=1$. All function spaces involved are spaces of complex valued functions. We will use the usual notation for various complex-valued function spaces: Lebesgue space $L^{r}=L^{r}(\mathbb{R}^{d})$, Sobolev spaces $H^{s,r}=H^{s,r}(\mathbb{R}^{d}):=(I-\Delta)^{-s/2}L^{r}$, homogeneous Sobolev spaces $\dot{H}^{s,r}=\dot{H}^{s,r}(\mathbb{R}^{d}):=(-\Delta)^{-s/2}L^{r}$, Besov spaces $B^{s}_{r,b}=B^{s}_{r,b}(\mathbb{R}^{d})$ and homogeneous Besov spaces $\dot{B}^{s}_{r,b}=\dot{B}^{s}_{r,b}(\mathbb{R}^{d})$. For the corresponding interpolation and embedding properties of these spaces, refer to \cite{B-L,T1}. For any interval $I\subset\mathbb{R}$ and any Banach space $X$ just mentioned, we denote by $C(I,X)$ the space of strongly continuous functions from $I$ to $X$ and by $L^{q}(I,X)$ the space of measurable functions $u$ from $I$ to $X$ such that $\|u(\cdot)\|_{X}\in L^{q}(I)$. Let $L^{q}_{t}(\mathbb{R},L^{r}_{x}(\mathbb{R}^{d}))$ denote the Banach space equipped with norm
\[\|u\|_{L^{q}(\mathbb{R},L^{r}(\mathbb{R}^{d}))}=(\int_{\mathbb{R}}(\int_{\mathbb{R}^{d}}|u(t,x)|^{r}dx)^{q/r}dt)^{1/q},\]
with the usual modifications when $q$ or $r$ is infinity, or when the domain $\mathbb{R}\times\mathbb{R}^{d}$ is replaced by a smaller region of space-time such as $I\times\mathbb{R}^{d}$.

As usual, we define the ``admissible pair" as below, which plays an important role in our space-time estimates.
\begin{defn}
We say that a pair $(q,r)$ is admissible if
\begin{equation}\label{admissible pair}
    \frac{2}{q}=\delta(r)=d(\frac{1}{2}-\frac{1}{r})
\end{equation}
and $2\leq r\leq\frac{2d}{d-2}$($2\leq r\leq\infty$ if $d=1$, $2\leq r<\infty$ if $d=2$).
\end{defn}

Throughout this paper, $(\gamma,\rho)$ denotes a particular choice of admissible pair defined by
\begin{equation}\label{choice}
    \gamma=\frac{2(d-2)}{d-4}, \,\,\,\,\, \rho=\frac{2d(d-2)}{d(d-4)+8}.
\end{equation}

The following definition concerns the parameters that occur in our paper.
\begin{defn}\label{parameter}
For any $2\leq r<d/2$, we define the corresponding Sobolev exponent $r^{\ast}\in(r,+\infty)$ by
\begin{equation}\label{Sobolev exponent}
    \frac{1}{r^{\ast}}=\frac{1}{r}-\frac{2}{d},
\end{equation}
so that $\dot{H}^{2,r}(\mathbb{R}^{d})\hookrightarrow L^{r^{\ast}}(\mathbb{R}^{d})$ by Sobolev embedding.
\end{defn}

In what follows positive constants will be denoted by $C$ and will change from line to line. If necessary, by $C_{\star,\cdots,\star}$ we denote positive constants depending only on the quantities appearing in subscript continuously.

\subsection{Littlewood-Paley theory on time variable.}
We define the Fourier transform with respect to time variable $t$ to be
\[\mathcal{F}f(\tau)=\widehat{f}(\tau):=\int_{\mathbb{R}_{t}}e^{-2\pi it\tau}f(t)dt.\]
Let $\wp$ be an annulus in momentum space $\mathbb{R}_{\tau}$ given by
\[\wp=\{\tau\in\mathbb{R}_{\tau}\,:\, \frac{3}{4}\leq |\tau|\leq\frac{8}{3}\},\]
and let $\chi(\tau)$ be a radial bump function supported in the ball $\{\tau\in\mathbb{R}_{\tau}: |\tau|\leq 4/3\}$ and equal to 1 on the ball $\{\tau\in\mathbb{R}_{\tau}: |\tau|\leq 3/4\}$, $\widehat{\varphi}(\tau)=\chi(\tau/2)-\chi(\tau)$ be a radial bump function supported in the annulus $\wp$ respectively, such that
\[0\leq\chi(\tau),\,\,\widehat{\varphi}(\tau)\leq 1.\]
So we can define decomposition of the whole momentum space $\mathbb{R}_{\tau}$,
\[\left\{
  \begin{array}{ll}
    \chi(\tau)+\sum_{j\geq 0}\widehat{\varphi}(2^{-j}\tau)=1, & \forall \tau \in \mathbb{R}_\tau  \,\,\,\,\, (inhomogeneous),\\
    \,&\\
   \sum_{j\in \mathbb{Z}}\widehat{\varphi}(2^{-j}\tau)=1, & \forall \tau \in \mathbb{R}_\tau\setminus \{0\} \,\,\, (homogeneous),
  \end{array}
\right.\]
with the following properties: \\
(i) $supp\,\widehat{\varphi}(2^{-j}\cdot)\bigcap supp\,\widehat{\varphi}(2^{-j'}\cdot)=\emptyset$,
for $\forall j,j'\in\mathbb{Z}$ such that $|j-j'|\geq 2$; \\
(ii) $supp\,\chi(\cdot)\bigcap supp\,\widehat{\varphi}(2^{-j}\cdot)=\emptyset$, for $\forall j\geq 1$.

Now setting $h(t)=\mathcal{F}^{-1}(\chi(\tau))$, $\varphi(t)=\mathcal{F}^{-1}(\widehat{\varphi}(\tau))$, then we can define the Fourier multipliers as following\\
$\left\{
\begin{array}{ll}
\Delta_{j}u=\mathcal{F}^{-1}(\widehat{\varphi}(2^{-j}\tau)\widehat{u}(\tau))
=2^{j}\int_{\mathbb{R}}\varphi(2^js)u(t-s)ds, &j\in\mathbb{Z},\\
\,&\\
S_{<j}u=\mathcal{F}^{-1}(\chi(2^{-j}\tau)\widehat{u}(\tau))=2^{j} \int_{\mathbb{R}}h(2^js)u(t-s)ds, &j\in\mathbb{Z},\\
\,&\\
S_{\geq j}u=u-S_{<j}u=\mathcal{F}^{-1}((1-\chi(2^{-j}\tau))\widehat{u}(\tau)), &j\in\mathbb{Z}.
\end{array}
\right.$ \\

Similarly, we can define $S_{\leq j}$, $S_{>j}$ and $S_{i\leq \cdot<j}:=S_{<j}-S_{<i}$. Due to identity $\widehat{\varphi}(\tau/2^j)=\chi(\tau/2^{j+1})-\chi(\tau/2^j)$, we have $\Delta_j u=(S_{<j+1}-S_{<j})u$ for any $j\in\mathbb{Z}$. So for any $u\in \mathcal{S}'(\mathbb{R}_{t})$, we have inhomogeneous Littlewood-Paley dyadic decomposition $u=S_{<0}u+\sum_{j\geq0}\Delta_{j}u$ and homogeneous decomposition $u=\sum_{j\in\mathbb{Z}}\Delta_{j}u$ in $\mathcal{S}'(\mathbb{R}_{t})$, respectively.

As with all Fourier multipliers, the Littlewood-Paley operators on time variable commute with differential operators such as $i\partial_{t}+\Delta$ and complex conjugation. They are self-adjoint and bounded on every $L_{t}^{p}$ and $\dot{H}_{t}^{s}$ space for $1\leq p\leq\infty$ and $s\geq0$. We will use basic properties of these operators frequently, for instance, in the following lemma(see \cite{T-V,T2}).
\begin{lem}\label{bernstein}(Bernstein estimates). For $1\leq p\leq q\leq \infty$, $s\geq0$, we have
\[\|\partial_{t}^{\pm s}\Delta_{j}f\|_{L^{p}_{t}(\mathbb{R})}\sim 2^{\pm js}\|\Delta_{j}f\|_{L^{p}_{t}(\mathbb{R})},\]
\[\|\partial_{t}^{s}S_{\leq j}f\|_{L^{p}_{t}(\mathbb{R})}\leq C2^{js}\|S_{\leq j}f\|_{L^{p}_{t}(\mathbb{R})},\]
\[\|\Delta_{j}f\|_{L^{q}_{t}(\mathbb{R})}\leq C2^{j(\frac{1}{p}-\frac{1}{q})}\|\Delta_{j}f\|_{L^{p}_{t}(\mathbb{R})},\]
\[\|S_{\leq j}f\|_{L^{q}_{t}(\mathbb{R})}\leq C2^{j(\frac{1}{p}-\frac{1}{q})}\|S_{\leq j}f\|_{L^{p}_{t}(\mathbb{R})},\]
\[\|S_{\geq j}f\|_{L^{p}_{t}(\mathbb{R})}\leq C2^{-js}\|\partial_{t}^{s}S_{\geq j}f\|_{L^{p}_{t}(\mathbb{R})}.\]
\end{lem}

\subsection{Properties of the Schr\"{o}dinger group $(e^{it\Delta})_{t\in\mathbb{R}}$}
We denote by $(e^{it\Delta})_{t\in\mathbb{R}}$ the Schr\"{o}dinger group, which is isometric on $H^{s}$ and $\dot{H}^{s}$ for every $s\geq0$, and satisfies the Dispersive estimates and Strichartz's estimates(for more details, see Keel and Tao \cite{K-T}). We will use freely the well-known properties of of the Schr\"{o}dinger group $(e^{it\Delta})_{t\in\mathbb{R}}$(see \cite{C} for an account of these properties). Here we only mention specially the Strichartz's estimates for non-admissible pair as below, which can be found in \cite{C,T-V}.
\begin {lem}\label{lem9}
(Strichartz's estimates for non-admissible pairs). Let $I$ be an interval of $\mathbb{R}$ (bounded or not), set $J=\bar{I}$, let $t_{0}\in J$, and consider $\Phi$ defined by: $t\mapsto \Phi_{f}(t)=\int^{t}_{t_{0}}e^{i(t-s)\Delta}f(s)ds$ for $t\in I$. Assume $2<r<2d/(d-2)$($2<r\leq\infty$ if $d=1$) and let $1<a,\tilde{a}<\infty$ satisfy \[\frac{1}{\tilde{a}}+\frac{1}{a}=\delta(r)=N(\frac{1}{2}-\frac{1}{r}).\]
It follows that $\Phi_{f}\in L^{a}(I,L^{r}(\mathbb{R}^{d}))$ for every $f\in L^{\tilde{a}'}(I,L^{r'}(\mathbb{R}^{d}))$. Moreover, there exists a constant C independent of $I$ such that
\[\|\Phi_{f}\|_{L^{a}(I,L^{r})}\leq C\|f\|_{L^{\tilde{a}'}(I,L^{r'})}\]
for every $f\in L^{\tilde{a}'}(I,L^{r'}(\mathbb{R}^{d}))$.
\end{lem}

\section{PROOF OF THEOREM \ref{th1}}

Now we are ready for proving Theorem \ref{th1}. First, let us recall some facts concerning the global existence, regularity and uniqueness of the solution to Cauchy problem \eqref{eq1}(see \cite{C-W}). To this end, we define
\begin{equation}\label{nonlinear interaction}
    \mathcal{G}(u)(t)=-i\int_{0}^{t}e^{i(t-s)\Delta}g(u(s))ds.
\end{equation}
From \eqref{nonlinearity0}, \eqref{nonlinearity1}, Sobolev embedding $\dot{H}^{2,r}(\mathbb{R}^{d})\hookrightarrow L^{r^{\ast}}(\mathbb{R}^{d})$, the equations $i\partial_{t}\mathcal{G}(u)+\Delta\mathcal{G}(u)=g(u)$, $\partial_{t}\mathcal{G}(u)(t)=-ie^{it\Delta}g(\phi)-i\int_{0}^{t}e^{i(t-s)\Delta}\partial_{s}g(u(s))ds$, Strichartz's and H\"{o}lder's estimates, it follows that
\begin{equation}\label{eq2}
    \|\partial_{t}\mathcal{G}(u)\|_{L^{q}(\mathbb{R},L^{r}(\mathbb{R}^{d}))}\leq C\|\Delta\phi\|_{L^{2}(\mathbb{R}^{d})}+C\|\partial_{t}g(u)\|_{L^{\gamma'}(\mathbb{R},L^{\rho'}(\mathbb{R}^{d}))},
\end{equation}
\begin{equation}\label{eq3}
    \|\Delta\mathcal{G}(u)\|_{L^{q}(\mathbb{R},L^{r})}\leq\|g(u)\|_{L^{q}(\mathbb{R},L^{r})}
    +C\|\Delta\phi\|_{L^{2}}+C\|\partial_{t}g(u)\|_{L^{\gamma'}(\mathbb{R},L^{\rho'})},
\end{equation}
\begin{equation}\label{eq4}
    \|g(u)\|_{L^{q}(\mathbb{R},L^{r}(\mathbb{R}^{d}))}\leq C\|u\|^{\frac{4}{d-4}}_{L^{\infty}(\mathbb{R},L^{\frac{2d}{d-4}}(\mathbb{R}^{d}))}\|\Delta u\|_{L^{q}(\mathbb{R},L^{r}(\mathbb{R}^{d}))},
\end{equation}
\begin{equation}\label{eq5}
    \|\partial_{t}g(u)\|_{L^{\gamma'}(\mathbb{R},L^{\rho'}(\mathbb{R}^{d}))}\leq C\|\Delta u\|^{\frac{4}{d-4}}_{L^{\gamma}(\mathbb{R},L^{\rho}(\mathbb{R}^{d}))}
    \|u_{t}\|_{L^{\gamma}(\mathbb{R},L^{\rho}(\mathbb{R}^{d}))},
\end{equation}
\begin{eqnarray}\label{eq6}
&&\|\mathcal{G}(u)-\mathcal{G}(v)\|_{L^{q}(\mathbb{R},L^{r}(\mathbb{R}^{d}))}\\
\nonumber &\leq& C(\|\Delta u\|^{\frac{4}{d-4}}_{L^{\gamma}(\mathbb{R},L^{\rho}(\mathbb{R}^{d}))}+\|\Delta v\|^{\frac{4}{d-4}}_{L^{\gamma}(\mathbb{R},L^{\rho}(\mathbb{R}^{d}))})
\|u-v\|_{L^{\gamma}(\mathbb{R},L^{\rho}(\mathbb{R}^{d}))},
\end{eqnarray}
where the constant $C$ is independent of $u$, $v$, $\phi$ and $g$, $(\gamma,\rho)$ is the admissible pair defined by \eqref{choice} and $(q,r)$ is an arbitrary admissible pair.

Therefore, if $\phi\in H^{2}$ with $\|\Delta\phi\|_{L^{2}}$ sufficiently small, by Strichartz's estimates, it follows that the problem \eqref{eq1} can be solved by using fixed point theorem to the equivalent integral equation(Duhamel's formula)
\begin{equation}\label{Duhamel}
    u(t)=\mathcal{H}(u)(t):=e^{it\Delta}\phi+\mathcal{G}(u)(t)
\end{equation}
in the set
\[E=\{u\in L^{\gamma}(\mathbb{R},H^{2,\rho})\cap H^{1,\gamma}(\mathbb{R},L^{\rho}); \|\Delta u\|_{L^{\gamma}(\mathbb{R},L^{\rho})}+\|u_{t}\|_{L^{\gamma}(\mathbb{R},L^{\rho})}\leq2\eta\}\]
equipped with the distance $d(u,v)=\|u-v\|_{L^{\gamma}(\mathbb{R},L^{\rho})}$ for $\eta>0$ sufficiently small. $(E,d)$ is a complete metric space(Indeed, the closed ball of radius $\eta$ in $L^{\gamma}(\mathbb{R},L^{\rho}(\mathbb{R}^{d}))$ is weakly compact). By Lemma 5.6 in \cite{C-W}, we have $g(u)\in C(\mathbb{R},L^{2})$ for any $u\in E$, moreover, there is a uniform estimate
\begin{eqnarray}\label{eq7}
&&\|g(u)\|_{L^{\infty}(\mathbb{R},L^{2})}\leq C\|u\|^{\frac{d}{d-4}}_{L^{\infty}(\mathbb{R},L^{\frac{2d}{d-4}})}\\
\nonumber &\leq& C[\|\Delta u\|^{\frac{d}{d-4}}_{L^{\gamma}(\mathbb{R},L^{\rho})}\|u_{t}\|_{L^{\gamma}(\mathbb{R},L^{\rho})}]^{\frac{d}{2(d-2)}}\leq C(2\eta)^{\frac{d}{d-4}}.
\end{eqnarray}

To see the details in the construction of global solutions to IVP \eqref{eq1}, let us define the norm $\dot{X}_{q,r}(\mathbb{R}_{t}\times\mathbb{R}_{x}^{d})$ for every admissible pair $(q,r)$ by
\begin{equation}\label{eq8}
    \|u\|_{\dot{X}_{q,r}}:=\|\Delta u\|_{L^{q}(\mathbb{R},L^{r})}+\|u_{t}\|_{L^{q}(\mathbb{R},L^{r})}.
\end{equation}
By applying Strichartz's estimates, we can deduce from \eqref{eq2}-\eqref{eq7} that
\begin{equation}\label{eq9}
    \|\mathcal{H}(u)\|_{L^{q}(\mathbb{R},L^{r}(\mathbb{R}^{d}))}\leq C\|\phi\|_{L^{2}(\mathbb{R}^{d})}+C(2\eta)^{\frac{4}{d-4}}\|u\|_{L^{\gamma}(\mathbb{R},L^{\rho}(\mathbb{R}^{d}))},
\end{equation}
\begin{equation}\label{eq10}
    \|\mathcal{H}(u)\|_{\dot{X}_{q,r}}\leq C\|\Delta \phi\|_{L^{2}}+C(2\eta)^{\frac{4}{d-4}}\|u\|_{\dot{X}_{q,r}}+C\|u\|^{\frac{d}{d-4}}_{\dot{X}_{\gamma,\rho}},
\end{equation}
\begin{equation}\label{eq11}
    \|\mathcal{G}(u)-\mathcal{G}(v)\|_{L^{q}(\mathbb{R},L^{r}(\mathbb{R}^{d}))}\leq 2C(2\eta)^{\frac{4}{d-4}}\|u-v\|_{L^{\gamma}(\mathbb{R},L^{\rho}(\mathbb{R}^{d}))}
\end{equation}
for arbitrary $u,v\in E$ and admissible pair $(q,r)$, where the constant $C$ is independent of $u$, $v$, $\phi$ and $g$.

Now let constant $K$ be larger than the constant $C$ appearing in the above three estimates \eqref{eq9}-\eqref{eq11} for the particular choice of the admissible pair $(q,r)=(\gamma,\rho)$. Thus if $\phi\in H^{2}(\mathbb{R}^{d})$ satisfies $K\|\Delta\phi\|_{L^{2}}<\eta$ with $\eta>0$ sufficiently small such that
\begin{equation}\label{eq12}
    4K(2\eta)^{\frac{4}{d-4}}\leq 1,
\end{equation}
then the estimates \eqref{eq9}-\eqref{eq11} imply that $\mathcal{H}: E\rightarrow E$ and $\mathcal{H}$ is a strict contraction on $(E,d)$. Therefore, $\mathcal{H}$ has a fixed point $u$, which is the unique global solution of \eqref{eq1} in $E$. Moreover, it follows from \eqref{eq2}, \eqref{eq5}, \eqref{Duhamel} and \eqref{eq9} that $u$, $u_{t}\in L^{q}(\mathbb{R},L^{r})\cap C(\mathbb{R},L^{2})$ for every admissible pair $(q,r)$. Since by Lemma 5.6 in \cite{C-W}, we have $g(u)\in C(\mathbb{R},L^{2})$ and uniform estimate \eqref{eq7}, the equation \eqref{eq1} implies that $\Delta u\in C(\mathbb{R},L^{2})$ and $u\in L^{\infty}(\mathbb{R},H^{2})$, the uniform estimate follows from \eqref{eq7} and \eqref{eq10}. One can also deduce easily from $u\in L^{\infty}(\mathbb{R},H^{2})$ and $u\in L^{\gamma}(\mathbb{R},H^{2,\rho})$ that $u\in L^{q}(\mathbb{R},H^{2,r})$ for every admissible pair $(q,r)$ with $r\leq\rho$. However, in fact, since $d\geq8$, let $\tilde{C}(r,d)$ denote the best constant in the Sobolev inequality
\[\|u\|_{L^{r^{\ast}}(\mathbb{R}^{d})}\leq C(r,d)\|\Delta u\|_{L^{r}(\mathbb{R}^{d})},\]
where $r\in[2,\frac{2d}{d-2}]$, then from $\sup_{r}\tilde{C}(r,d)=C(d)<\infty$, \eqref{eq3}, \eqref{eq4}, \eqref{eq7} and \eqref{eq10}, we can infer that by possibly choosing $\eta$ smaller if necessary, $u\in L^{q}(\mathbb{R},H^{2,r})$ for every admissible pair $(q,r)$. The unconditional uniqueness for solutions to \eqref{eq1} in $C(\mathbb{R},H^{2}(\mathbb{R}^{d}))$ follows from \cite{C,K3}(see e.g. Proposition 4.2.13 in \cite{C}).

Now we consider the continuity of the the solution map $\phi\mapsto u$ in $H^{2}(\mathbb{R}^{d})$ for Cauchy problem \eqref{eq1} under the restriction $8\leq d<12$. Suppose $\phi_{n}\rightarrow\phi$ in $H^{2}(\mathbb{R}^{d})$, then we have $K\|\Delta\phi_{n}\|_{L^{2}}<\eta$ for $n$ sufficiently large, thus the corresponding solutions $u_{n}$ of \eqref{eq1} with initial value $\phi_{n}$ are global for $n$ large enough and $u_{n}\in L^{q}(\mathbb{R},H^{2,r})$, $\partial_{t}u_{n}\in L^{q}(\mathbb{R},L^{r})$ for every admissible pair $(q,r)$. Furthermore, there exists $n_{0}\in\mathbb{N}$ large enough such that
\begin{equation}\label{eq13}
    \max\{\|u\|_{\dot{X}_{\gamma,\rho}},\sup_{n\geq n_{0}}\|u_{n}\|_{\dot{X}_{\gamma,\rho}}\}\leq 2\eta,
\end{equation}
in addition, by possibly choosing $\eta$ smaller if necessary, we can deduce from Strichartz's estimates and \eqref{eq10} that for arbitrary admissible pair $(q,r)$,
\begin{equation}\label{eq14}
    \max\{\|u\|_{\dot{X}_{q,r}},\sup_{n\geq n_{0}}\|u_{n}\|_{\dot{X}_{q,r}}\}\leq C_{q,r}\eta.
\end{equation}

\newtheorem{clm}[thm]{Claim}

\begin{clm}\label{claim}
Assume $8\leq d<12$, we claim that by possibly choosing $\eta$ smaller if necessary, we have as $n\rightarrow\infty$,
\begin{equation}\label{critical-1}
u_{n}\rightarrow u \quad in \quad L^{\gamma}(\mathbb{R},L^{\rho^{\ast}}(\mathbb{R}^{d})).
\end{equation}
\end{clm}
\begin{proof}
We will prove our claim in the spirit of Tao and Visan \cite{T-V}. We have to avoid taking full time derivative, since this is what turns the nonlinearity from Lipschitz into just H\"{o}lder continuous of order $\frac{4}{d-4}$, we need to take fewer than $\frac{4}{d-4}$ but more than $1/2$ time derivatives instead. To this end, we will use the norms $U=U(\mathbb{R}\times\mathbb{R}^{d})$ and $V=V(\mathbb{R}\times\mathbb{R}^{d})$ defined by
\begin{equation}\label{exotic norms 1}
    \|u\|_{U}=\{\sum_{k\in\mathbb{Z}}2^{2sk}\|\Delta_{k}(\|u\|_{L_{x}^{\tau}(\mathbb{R}^{d})})\|_{L_{t}^{a}(\mathbb{R})}^{2}\}^{1/2},
\end{equation}
\begin{equation}\label{exotic norms 2}
    \|f\|_{V}=\{\sum_{k\in\mathbb{Z}}2^{2sk}\|\Delta_{k}(\|f\|_{L_{x}^{\tau'}(\mathbb{R}^{d})})\|_{L_{t}^{\tilde{a}'}(\mathbb{R})}^{2}\}^{1/2},
\end{equation}
which require roughly $s$ degrees of differentiability, where $s=\frac{20}{3(d+1)}$, $\tau=\frac{2(d+1)}{d-1}$, $a=\frac{6(d+1)}{34-3d}$ and $\tilde{a}'=\frac{6(d+1)}{40-3d}$. Choosing a particular admissible pair $(q_{1},r_{1})=(\frac{4(d+1)}{d-4},\frac{2d(d+1)}{d^{2}+4})$, we will also need the Strichartz space $\dot{W}=\dot{W}(\mathbb{R}\times\mathbb{R}^{d})$ defined as the closure of the test functions under the norm
\begin{equation}\label{Strichartz norms}
    \|u\|_{\dot{W}(\mathbb{R}\times\mathbb{R}^{d})}:=\|u\|_{\dot{X}_{q_{1},r_{1}}(\mathbb{R}\times\mathbb{R}^{d})}.
\end{equation}

As the Littlewood-Paley operators on time commute with time derivatives of arbitrary order, similar to the Strichartz's estimates for non-admissible pairs(Lemma \ref{lem9}), we can deduce from dispersive estimates and Riesz potential inequalities(refer to \cite{S,T2}) that, there exists a constant $C>0$ such that for any $f\in V$,
\begin{equation*}
    \|\Delta_{k}(\|\int^{t}_{t_{0}}e^{i(t-s)\Delta}f(s)ds\|_{L_{x}^{\tau}(\mathbb{R}^{d})})\|_{L_{t}^{a}(\mathbb{R})}\leq C\|\Delta_{k}(\|f\|_{L_{x}^{\tau'}(\mathbb{R}^{d})})\|_{L_{t}^{\tilde{a}'}(\mathbb{R})}.
\end{equation*}
Squaring the above inequality, multiplying by $2^{2sk}$, and summing over all integer $k$'s, we obtain the exotic Strichartz estimate, that is, for any $f\in V$,
\begin{equation}\label{exotic Strichartz estimate}
    \|\int^{t}_{t_{0}}e^{i(t-s)\Delta}f(s)ds\|_{U}\leq C\|f\|_{V}.
\end{equation}

Note that $\|\cdot\|_{U}\sim\|\cdot\|_{\dot{B}^{s}_{a,2}(\mathbb{R},L_{x}^{\tau}(\mathbb{R}^{d}))}$, $\|\cdot\|_{V}\sim\|\cdot\|_{\dot{B}^{s}_{\tilde{a}',2}(\mathbb{R},L_{x}^{\tau'}(\mathbb{R}^{d}))}$ and $0<s=\frac{20}{3(d+1)}<\frac{4}{d-4}$, thus similar to the nonlinear estimates $$\|g_{z}(v)u\|_{\dot{B}^{s}_{\tilde{a}',2}}\leq C\|v\|^{\frac{4}{d-4}}_{L^{q_{1}}}\|u\|_{\dot{B}^{s}_{a,2}}$$ 
obtained in \cite{C-D-Y,C-F-H}, we would expect a nonlinear estimate as follow:
\begin{equation}\label{expected estimate}
    \|g_{z}(v)u\|_{V}\leq C\|v\|_{\dot{W}(\mathbb{R}\times\mathbb{R}^{d})}^{\frac{4}{d-4}}\|u\|_{U},
\end{equation}
whenever the right-hand side makes sense.
 
To show \eqref{expected estimate}, first we prove the following frequency-localized nonlinear estimate for any $k\in\mathbb{Z}$,
\begin{eqnarray}\label{f-l nonlinear}
&&\|\Delta_{k}(\|g_{z}(v)u\|_{L_{x}^{\tau'}(\mathbb{R}^{d})})\|_{L^{\tilde{a}'}_{t}(\mathbb{R})}\\
\nonumber &\leq&C\|v\|^{\frac{4}{d-4}}_{\dot{W}(\mathbb{R}\times\mathbb{R}^{d})}\sum_{l\in\mathbb{Z}}\min\{1,2^{\frac{4(l-k)}{d-4}}\}
\|\Delta_{l}(\|u\|_{L_{x}^{\tau}(\mathbb{R}^{d})})\|_{L^{a}_{t}(\mathbb{R})}.
\end{eqnarray}

Indeed, by scaling, we only need to show (\ref{f-l nonlinear}) for $k=0$. By Minkowski's and H\"{o}lder's inequality, we have
\begin{equation}\label{F-L11}
\|\Delta_{0}(\|g_{z}(v)u\|_{L_{x}^{\tau'}})\|_{L^{\tilde{a}'}_{t}}\leq \sum_{l\in\mathbb{Z}}\|\Delta_{0}\{\|g_{z}(v)\|_{L_{x}^{\frac{(d-4)r_{1}^{\ast}}{4}}}\Delta_{l}(\|u\|_{L_{x}^{\tau}})\}\|_{L^{\tilde{a}'}_{t}}.
\end{equation}
To bound the right-hand side of (\ref{F-L11}), on one hand, by using (\ref{nonlinearity1}), (\ref{Strichartz norms}) and H\"{o}lder's inequality, note that $\dot{H}^{2,r_{1}}(\mathbb{R}^{d})\hookrightarrow L^{r_{1}^{\ast}}(\mathbb{R}^{d})$, we get
\begin{equation}\label{F-L12}
\sum_{l\geq-2}\|\Delta_{0}\{\|g_{z}(v)\|_{L_{x}^{\frac{(d-4)r_{1}^{\ast}}{4}}}\Delta_{l}(\|u\|_{L_{x}^{\tau}})\}\|_{L^{\tilde{a}'}_{t}}
\leq C\|v\|^{\frac{4}{d-4}}_{\dot{W}}\sum_{l\geq-2}\|\Delta_{l}(\|u\|_{L_{x}^{\tau}})\|_{L^{a}_{t}};
\end{equation}
on the other hand, by H\"{o}lder's and Bernstein estimates(see Lemma \ref{bernstein}), we have
\begin{eqnarray}\label{F-L13}
&&\sum_{l\leq-3}\|\Delta_{0}\{\|g_{z}(v)\|_{L_{x}^{\frac{(d-4)r_{1}^{\ast}}{4}}(\mathbb{R}^{d})}\Delta_{l}(\|u\|_{L_{x}^{\tau}})\}\|_{L^{\tilde{a}'}_{t}(\mathbb{R})}\\
\nonumber &\leq&C\|S_{\geq-1}(\|g_{z}(v)\|_{L_{x}^{\frac{(d-4)r_{1}^{\ast}}{4}}(\mathbb{R}^{d})})\|_{L^{q}_{t}(\mathbb{R})}\sum_{l\leq-3}\|\Delta_{l}(\|u\|_{L_{x}^{\tau}(\mathbb{R}^{d})})\|_{L^{p}_{t}(\mathbb{R})}\\
\nonumber &\leq&C\|g_{z}(v)\|_{L^{\frac{(d-4)q_{1}}{4}}_{t}(\mathbb{R},L_{x}^{\frac{(d-4)r_{1}^{\ast}}{4}}(\mathbb{R}^{d}))}\sum_{l\leq-3}2^{\frac{4l}{d-4}}
\|\Delta_{l}(\|u\|_{L_{x}^{\tau}(\mathbb{R}^{d})})\|_{L^{a}_{t}(\mathbb{R})},
\end{eqnarray}
where $\frac{1}{p}=\frac{34-3d}{6(d+1)}-\frac{4}{d-4}$ and $q=\frac{(d+1)(d-4)}{5d}$. By \eqref{nonlinearity1}, H\"{o}lder's inequality and Sobolev embedding $\dot{H}^{2,r_{1}}\hookrightarrow L^{r_{1}^{\ast}}$, we get
\begin{equation}\label{F-L14}
    \|g_{z}(v)\|_{L^{\frac{(d-4)q_{1}}{4}}_{t}(\mathbb{R},L_{x}^{\frac{(d-4)r_{1}^{\ast}}{4}}(\mathbb{R}^{d}))}\leq C\|v\|^{\frac{4}{d-4}}_{\dot{W}(\mathbb{R}\times\mathbb{R}^{d})}.
\end{equation}
By combining the estimates \eqref{F-L11}-\eqref{F-L14}, we infer that
\begin{eqnarray}\label{F-L15}
&&\|\Delta_{0}(\|g_{z}(v)u\|_{L_{x}^{\tau'}(\mathbb{R}^{d})})\|_{L^{\tilde{a}'}_{t}(\mathbb{R})}\\
\nonumber &\leq&C\|v\|^{\frac{4}{d-4}}_{\dot{W}(\mathbb{R}\times\mathbb{R}^{d})}\sum_{l\in\mathbb{Z}}\min\{1,2^{\frac{4l}{d-4}}\}
\|\Delta_{l}(\|u\|_{L_{x}^{\tau}(\mathbb{R}^{d})})\|_{L^{a}_{t}(\mathbb{R})}.
\end{eqnarray}
It follows from \eqref{F-L15} that \eqref{f-l nonlinear} holds for $k=0$, then by scaling, we know it holds for any $k\in\mathbb{Z}$, this completes the proof of \eqref{f-l nonlinear}.

We can rewrite \eqref{f-l nonlinear} as
\begin{eqnarray}\label{re f-l nonlinear}
&&2^{sk}\|\Delta_{k}(\|g_{z}(v)u\|_{L_{x}^{\tau'}(\mathbb{R}^{d})})\|_{L^{\tilde{a}'}_{t}(\mathbb{R})}\\
\nonumber &\leq&C\|v\|^{\frac{4}{d-4}}_{\dot{W}(\mathbb{R}\times\mathbb{R}^{d})}\sum_{l\in\mathbb{Z}}\min\{2^{s(k-l)},2^{(l-k)(\frac{4}{d-4}-s)}\}
2^{sl}\|\Delta_{l}(\|u\|_{L_{x}^{\tau}(\mathbb{R}^{d})})\|_{L^{a}_{t}(\mathbb{R})}.
\end{eqnarray}
for any $k\in\mathbb{Z}$, where $s=\frac{20}{3(d+1)}$. Therefore, we deduce from \eqref{exotic norms 1}, \eqref{exotic norms 2}, \eqref{re f-l nonlinear} and Schur's test the expected nonlinear estimate as follow:
\begin{equation}\label{nonlinear estimate}
    \|g_{z}(v)u\|_{V}\leq C\|v\|_{\dot{W}(\mathbb{R}\times\mathbb{R}^{d})}^{\frac{4}{d-4}}\|u\|_{U},
\end{equation}
whenever the right-hand side makes sense. A similar statement holds with $g_{z}$ replaced by $g_{\bar{z}}$, and the corresponding proof is identical.

From Proposition 2.5 in \cite{P}, we have the following linear estimate:
\begin{equation}\label{linear}
    \|e^{it\Delta}\phi\|_{\dot{B}^{s_{1}}_{q,2}(\mathbb{R},\dot{H}^{s_{2},r}(\mathbb{R}^{d}))}\leq C\|\phi\|_{\dot{H}^{2s_{1}+s_{2},2}(\mathbb{R}^{d})},
\end{equation}
where $d\geq3$, $0\leq s_{1}<1$, $s_{2}\geq0$ and $(q,r)$ is an arbitrary admissible pair. Now let us define $\tilde{\tau}:=\frac{6d(d+1)}{3d(d+3)-68}$. Note that $(a,\tilde{\tau})$ is an admissible pair and $\dot{H}^{2(1-s),\tilde{\tau}}(\mathbb{R}^{d})\hookrightarrow L^{\tau}(\mathbb{R}^{d})$, we deduce from exotic Strichartz estimate (\ref{exotic Strichartz estimate}) combined with nonlinear estimate (\ref{nonlinear estimate}) that
\begin{eqnarray*}
  \|u_{n}-u\|_{U}&\leq& C\|e^{it\Delta}(\phi-\phi_{n})\|_{\dot{B}^{s}_{a,2}(\mathbb{R},\dot{H}^{2(1-s),\tilde{\tau}}(\mathbb{R}^{d}))}
  +C\|g(u_{n})-g(u)\|_{V}\\
&\leq&C\|\phi_{n}-\phi\|_{\dot{H}^{2}(\mathbb{R}^{d})}+C(\|u_{n}\|^{\frac{4}{d-4}}_{\dot{W}}+\|u\|^{\frac{4}{d-4}}_{\dot{W}})\|u_{n}-u\|_{U}.
\end{eqnarray*}
Thus by (\ref{eq14}) and \eqref{Strichartz norms}, we have for all sufficiently large $n\geq n_{0}$,
\[\|u_{n}-u\|_{U}\leq C\|\phi_{n}-\phi\|_{\dot{H}^{2}(\mathbb{R}^{d})}+C\eta^{\frac{4}{d-4}}\|u_{n}-u\|_{U}.\]
Therefore, by possibly choosing $\eta$ smaller such that
\begin{equation}\label{delta2}
    C\eta^{\frac{4}{d-4}}<1/2,
\end{equation}
we can deduce that as $n\rightarrow \infty$, $\|u_{n}-u\|_{U}\rightarrow 0$.

Since $\|\cdot\|_{U}\sim\|\cdot\|_{\dot{B}^{s}_{a,2}(\mathbb{R},L_{x}^{\tau}(\mathbb{R}^{d}))}$, we deduce from \eqref{eq14} and interpolation inequalities that
\begin{equation}\label{convergence}
    \lim_{n\rightarrow\infty}\|u_{n}-u\|_{L^{\gamma}(\mathbb{R},L^{\rho^{\ast}}(\mathbb{R}^{d}))}=0,
\end{equation}
this concludes our proof.
\end{proof}
Note that we have the formula
\[\partial_{t}(g(u_{n})-g(u))=g'(u_{n})\partial_{t}(u_{n}-u)+(g'(u_{n})-g'(u))\partial_{t}u,\] 
where $g'$ denotes the complex partial derivatives $g_{z}$ and $g_{\bar{z}}$, so it follows from \eqref{nonlinearity1}, \eqref{nonlinearity2} and H\"{o}lder's inequality that
\begin{equation}\label{eq15}
    \|\partial_{t}(g(u_{n})-g(u))\|_{L_{t}^{\gamma'}L_{x}^{\rho'}}\leq C\|u_{n}\|^{\frac{4}{d-4}}_{\dot{X}_{\gamma,\rho}}\|u_{n}-u\|_{\dot{X}_{\gamma,\rho}}+C\|u\|_{\dot{X}_{\gamma,\rho}}
    \|u_{n}-u\|^{\frac{4}{d-4}}_{L_{t}^{\gamma}L_{x}^{\rho^{\ast}}},
\end{equation}
where all space-time norms are on $\mathbb{R}\times\mathbb{R}^{d}$. Therefore, by Strichartz's estimates, \eqref{eq2}-\eqref{eq13} and \eqref{eq15}, we have
\begin{equation}\label{eq16}
    \|u_{n}-u\|_{L^{q}(\mathbb{R},L^{r}(\mathbb{R}^{d}))}\leq C\|\phi_{n}-\phi\|_{L^{2}(\mathbb{R}^{d})}+C\eta^{\frac{4}{d-4}}\|u_{n}-u\|_{L^{\gamma}(\mathbb{R},L^{\rho}(\mathbb{R}^{d}))};
\end{equation}
\begin{eqnarray}\label{eq17}
&&\|u_{n}-u\|_{\dot{X}_{q,r}}\leq C\|\phi_{n}-\phi\|_{\dot{H}^{2}(\mathbb{R}^{d})}+\|g(u_{n})-g(u)\|_{L_{t}^{q}L_{x}^{r}(\mathbb{R}\times\mathbb{R}^{d})}\\
\nonumber &&+C\|g(\phi_{n})-g(\phi)\|_{L^{2}(\mathbb{R}^{d})}+C\|\partial_{t}(g(u_{n})-g(u))\|_{L_{t}^{\gamma'}L_{x}^{\rho'}(\mathbb{R}\times\mathbb{R}^{d})}\\
\nonumber &\leq&C(1+\eta^{\frac{4}{d-4}})\|\phi_{n}-\phi\|_{\dot{H}^{2}(\mathbb{R}^{d})}+C\eta^{\frac{4}{d-4}}(\|u_{n}-u\|_{\dot{X}_{q,r}}
+\|u_{n}-u\|_{\dot{X}_{\gamma,\rho}})\\
\nonumber &&+C\eta\|u_{n}-u\|^{\frac{4}{d-4}}_{L_{t}^{\gamma}L_{x}^{\rho^{\ast}}(\mathbb{R}\times\mathbb{R}^{d})}.
\end{eqnarray}
Now let constant $K_{0}$ be larger than the constant $C$ appearing in the above three estimates \eqref{eq16} and \eqref{eq17} for the particular choice of the admissible pair $(q,r)=(\gamma,\rho)$. Then by possibly choosing $\eta$ smaller such that 
\begin{equation}\label{eq18}
    4K_{0}\eta^{\frac{4}{d-4}}<1,
\end{equation}
we can deduce from $\phi_{n}\rightarrow\phi$ in $H^{2}(\mathbb{R}^{d})$ and Claim \ref{claim} that \[\|u_{n}-u\|_{L^{\gamma}(\mathbb{R},L^{\rho}(\mathbb{R}^{d}))}\rightarrow0 \,\,\,\,\,\, and \,\,\,\,\,\, \|u_{n}-u\|_{\dot{X}_{\gamma,\rho}(\mathbb{R}\times\mathbb{R}^{d})}\rightarrow0,\]
as $n\rightarrow\infty$. The convergence for arbitrary admissible pair $(q,r)$ follows from Strichartz's estimates, or more precisely, from \eqref{eq16} and \eqref{eq17}.

In a word, we have proved for $8\leq d<12$ that, if $\phi\in H^{2}(\mathbb{R}^{d})$ with $\|\phi\|_{\dot{H}^{2}}$ sufficiently small and if $\phi_{n}\rightarrow \phi$ in $H^{2}(\mathbb{R}^{d})$, $u_{n}$ denotes the solution of \eqref{eq1} with the initial value $\phi_{n}$, then $u_{n}$ is global for $n$ large enough, furthermore, $u_{n}\rightarrow u$ in $L^{q}(\mathbb{R},H^{2,r}(\mathbb{R}^{d}))$ and $(u_{n})_{t}\rightarrow u_{t}$ in $L^{q}(\mathbb{R},L^{r}(\mathbb{R}^{d}))$ as $n\rightarrow\infty$ for every admissible pair $(q,r)$. In particular, $u_{n}\rightarrow u$ in $C(\mathbb{R},H^{2}(\mathbb{R}^{d}))$.\\

This completes the proof of Theorem \ref{th1}.\\

{\bf Acknowledgements:} The author's research was done during his visit at University of California, Berkeley, which is supported by a Young Researcher's Fellowship of AMSS, Chinese Academy of Sciences. \\

\end{document}